\documentclass[12pt, letterpaper]{article}
\usepackage[OT4]{fontenc}
\usepackage{graphics,graphicx}
\usepackage{amsmath}
\usepackage{amssymb,amsfonts}
\usepackage {amsthm}
\usepackage{xcolor}

\usepackage{enumerate}
\usepackage{hyperref}
\usepackage[T2A]{fontenc}
\usepackage[cp1251]{inputenc}
\usepackage[all]{xy}
\usepackage[linesnumbered,boxed]{algorithm2e}

\newtheorem{theorem}{Theorem}
\newtheorem{proposition}{Proposition}
\newtheorem{lemma}{Lemma}

\newcommand\myfootnote[1]{
\renewcommand{\thefootnote}{}
\footnotetext{#1}
\def\thefootnote{\@arabic\c@footnote}
}

\newcommand{\zero}{\mbox{\normalfont 0}}
\newcommand{\rvline}{\hspace*{-\arraycolsep}\vline\hspace*{-\arraycolsep}}

\renewcommand{\subsection}{\@startsection{subsection}{2}{0mm}{-\baselineskip}{-5pt}{\it \bf}}
\makeatother

\title{Ideals of general linear Lie  algebras of infinite-dimensional  vector spaces}
\author{Oksana Bezushchak, Waldemar Ho{\l}ubowski, Bogdana Oliynyk\footnote{The third author   was partially supported by the grant for scientific researchers from the ``Povir u sebe'' Ukrainian Foundation.}}

\begin{document}
	
	\maketitle
	
\small \noindent Faculty of Mechanics and Mathematics,
Taras Shevchenko National University of Kyiv, Ukraine \\
Faculty of Applied Mathematics, Silesian University of Technology, Poland \\ Department of Mathematics, National University
of Kyiv-Mohyla Academy, Ukraine

bezushchak@knu.ua, w.holubowski@polsl.pl, oliynyk@ukma.edu.ua


{\bf 2020}{ {\bf Mathematics Subject Classification:} 17B60, 17B65}

\begin{abstract}
Let $V$   be an infinite-dimensional  vector space over a field of characteristic not equal to $2.$ We classify ideals of  the Lie algebra   $\mathfrak{gl}(V)$ of all linear transformations of the space $V.$
\end{abstract}

\section*{Introduction}

Let $\mathbb{F}$ be a field of characteristic not equal to $2$. Any associative algebra $A$ over the field  $\mathbb{F}$   gives rise to the Lie algebra $A^{(-)} = (A, [a, b] = ab-ba)$. Let $V$   be an infinite-dimensional  vector space over $\mathbb{F}$. In this paper, we consider  the algebras $\text{End}_{\mathbb{F}}(V)$ of all linear transformations $V \to V$ and $\mathfrak{gl}(V)=\text{End}_{\mathbb{F}}(V)^{(-)}$ the general Lie algebra  of $V.$ 

For a cardinal  $\alpha \le \dim_{\mathbb{F}} V$ denote by $I_{\alpha}$ the ideal of all linear transformations $\varphi :V \rightarrow V$ such that $\dim_{\mathbb{F}} \varphi(V)$ is less than $\alpha$.  In particular, for the countable cardinal $\aleph_0$ the ideal $I_{\aleph_0}$ consists of all linear transformations of finite ranges.  Let $\text{Id}_V$ denote the identity transformation.

 N.~Jacobson \cite{Jacobson} classified ideals of the associative ring $\text{End}_{\mathbb{F}}(V)$ of linear transformations of $V.$ He proved that:

\emph{all ideals of the ring $\text{\emph{End}}_{\mathbb{F}}(V)$ of linear transformations of an infinite-dimensional vector space $V$ are: $(0),$  $I_{\alpha},$ where $\aleph_0 \le \alpha \le \dim_{\mathbb{F}} V, $ and $\text{\emph{End}}_{\mathbb{F}}(V).$}

The purpose of this paper is the classification of ideals of the algebra  $\mathfrak{gl}(V).$

\begin{theorem} 	\label{Thm_ideals}
Let $V$   be an infinite-dimensional  vector space over a field $\mathbb{F}$ of characteristic not equal to $2.$	Every  ideal of $\mathfrak{gl}(V)$ belongs to one of the following families:
	\begin{enumerate}[$(1)$]
		\item
		$\{0\},$ $\mathbb{F} \cdot \emph{Id}_{V},$ $\mathfrak{gl}(V);$
		\item
		$I_{\alpha}$ or $\mathbb{F} \cdot \emph{Id}_{V}+ I_{\alpha},$ where $\aleph_0 \le \alpha \le \dim_{\mathbb{F}} V ;$
		\item
		
		any subspase $U,$
		$$[I_{\aleph_0},\mathfrak{gl}(V)]\subseteq U \subseteq \mathbb{F} \cdot \emph{Id}_{V}+ I_{\aleph_0}.$$
		
	\end{enumerate}	
\end{theorem}

Moreover, the co-dimension of $[I_{\aleph_0},\mathfrak{gl}(V)]$ in $\mathbb{F} \cdot \text{Id}_{V}+ I_{\aleph_0}$ is $2.$ Therefore, if the field $\mathbb{F}$ is infinite then the family $(3)$ is an infinite family of ideals.

I.~Penkov and V.~Serganova \cite{Penkov_Serganova} classified ideals of countable-dimensional Mackey Lie algebras and I.~Penkov and A.~Petukhov \cite{Penkov_Petukhov} studied ideals of the enveloping algebra of an infinite-dimensional Lie algebra  viewed as the union of a chain of embeddings of simple finite-dimensional Lie algebras. The paper \cite{Hol-inf-Lie} contains the description of ideals of the algebra $\mathfrak{gl}(V)$ for a  countable-dimensional vector space~$V$. 

In \cite{14_B}, it was shown that the Lie algebra $\mathfrak{gl}(V)$ of an infinite-dimensional vector space $V$ is perfect, i.e. $\mathfrak{gl}(V)= [\mathfrak{gl}(V),\mathfrak{gl}(V)].$

A.~Rosenberg \cite{Rosenberg} (see also \cite{Holubowski_M_Z}) studied normal subgroups of the group $GL(V)$ of invertible  linear transformations on   $V.$

For some results on the structure of other associative and Lie algebras of infinite matrices see \cite{BaranovBahtZal,Baranov2,Baranov_Strade,14,BezOl,BezOl_2}.

\section{Ideals of the algebra $\mathfrak{gl}(V)$}

Let $U$ be a subset of an associative algebra  $A.$ Denote by $\text{id}_A(U)$ the ideal of the algebra $A,$ which is generated $U.$

\begin{lemma}[Alahmadi-Alsulami, \cite{Alahm}]	\label{lemma 2} Let $A$ be an associative algebra over a field $\mathbb{F}$ of characteristic  not equal to $2$  and let $U$ be  an  ideal of the Lie algebra $A^{(-)}.$ Then
	$$[\text{\emph{id}}_A([U,U]), A]\subseteq U.$$
	
\end{lemma}	
 An associative algebra $A$ is called a {\it prime} algebra if for any nonzero ideals $I, J \in A$ the product $I \cdot J$ is also nonzero.

The following lemma is due to I.~Herstein \cite{Herst},

 \begin{lemma}[\cite{Herst}] 	\label{lemma3} 
	Let $A$ be a unital associative prime algebra over a field of  characteristic $\ne 2.$ Let $U$ be an  ideal in $A^{(-)}$ such that $[U,U]=(0)$.   Then $U $ lies in the center of $A$.
\end{lemma}

 \begin{lemma}
	\label{lemma4} The factor-algebra  $\text{\emph{End}}_{\mathbb{F}}(V)/{I_{\alpha}}$ is prime  for any $\alpha$, $\aleph_0 \le \alpha \le \dim_{\mathbb{F}} V $.
\end{lemma}

\begin{proof} Suppose that the algebra  $\text{End}_{\mathbb{F}}(V)/{I_{\alpha}}$ has two proper ideals whose product is $(0)$. N.~Jacobson  \cite{Jacobson} proved that an arbitrary proper ideal of the algebra $\text{End}_{\mathbb{F}}(V)$ looks as $I_{\alpha},$ $\aleph_0 \le \alpha \le \dim_{\mathbb{F}} V.$ Hence, by Jacobson's Theorem these ideals are ${I_{\beta}}/{I_{\alpha}}$ and  ${I_{\gamma}}/{I_{\alpha}}$, where
	$\alpha <\beta,\gamma \leq  \dim_{\mathbb{F}} V$. If $\beta \le \gamma$ then $I_{\beta } \subseteq I_{\gamma}$ and
	$$I_{\beta}^2\subseteq I_{\alpha}.$$
	There exists a subspace  $W \subset V$, $\dim_{\mathbb{F}}W=\alpha.$  Let $\rho: \ V \to W$ be a projection of $V$ onto $W$. Then $\rho \in I_{\beta} \setminus I_{\alpha} $, $\rho ^2=\rho,$ a contradiction.
\end{proof}

\begin{proposition}
	\label{prop5} The center $C$  of the algebra $\text{\emph{End}}_{\mathbb{F}}(V)/{I_{\alpha}}$ is $$(\mathbb{F} \cdot \text{Id}_V + I_{\alpha})/{I_{\alpha}},$$ where $\alpha$, $\aleph_0 \le \alpha \le \dim_{\mathbb{F}} V .$
\end{proposition}

We need several lemmas to prove this  proposition. Let $z \in \text{End}_{\mathbb{F}}(V)$ be an element not lying in  $I_{\alpha}$ and such that $$ [z,\text{End}_{\mathbb{F}}(V)] \subseteq I_{\alpha}.$$

\begin{lemma} 	\label{ker} $\dim_{\mathbb{F}} \ker z < \alpha.$	
\end{lemma}

\begin{proof} The subspace $I=I_{\alpha}+\text{End}_{\mathbb{F}}(V) z$ of $\text{End}_{\mathbb{F}}(V)$ is an ideal. The ideal $I$ contains $z$ and, therefore, is strictly larger then $I_{\alpha}$. By Jacobson's Theorem, all ideals of $\text{End}_{\mathbb{F}}(V)$ are of the types $(0),$  $I_{\beta}$, where $\aleph_0 \le \beta \le \dim_{\mathbb{F}}V$, and $\text{End}_{\mathbb{F}}(V)$. 

\vspace{5pt}

\underline{Case 1.} $I_{\alpha}+\text{End}_{\mathbb{F}}(V) z=I_{\beta}$, $\alpha < \beta \le \dim_{\mathbb{F}}V$.

Suppose that $\dim_{\mathbb{F}} \ker z \ge \alpha$. Then there exists a subspace $V'\subseteq \ker z$ such that $\dim_{\mathbb{F}}V' = \alpha$.

Let $p$ be a projection $p: V \to V',$ that is $p(V)\subseteq V'$ and $p(v)=v$ for every element $v \in V'$. Then $p \in I_{\beta} $.  Hence, there exist elements $\varphi \in I_{\alpha}$ and $a \in \text{End}_{\mathbb{F}}(V)$ such that $p=\varphi+az.$ For an arbitrary element $v \in V'$ we have
$$v=p(v)=\varphi(v)+ a(z(v)).$$
Since $V'\subseteq \ker z$, it follows that $z(v)=0$. Hence, $\varphi(v)=v$ for all elements of $V'$. Therefore, $\dim_{\mathbb{F}}\varphi(V) \ge \alpha$, which contradicts the inclusion $\varphi \in I_{\alpha}$.

\vspace{5pt}

\underline{Case 2.} $I_{\alpha}+\text{End}_{\mathbb{F}}(V) z=\text{End}_{\mathbb{F}}(V).$

Recall that $\text{Id}_V$ is the identity transformation on $V$. Again there exist elements  $\varphi \in I_{\alpha}$ and $a \in \text{End}_{\mathbb{F}}(V)$ such that
$$\varphi(v)+ az=\text{Id}_V.$$
If $v \in \ker z$ then
$$v=\varphi(v)+ a(z(v))=\varphi(v),$$
hence $\ker z \subseteq \varphi(V) $. Since $\varphi \in I_{\alpha}$ it implies that $\dim_{\mathbb{F}} \ker z < \alpha$.
\end{proof}

\begin{lemma}
	\label{subs}
	If $W$ is a subspace of $V$ and $$W \cap z(W)=(0)$$ then $\dim_{\mathbb{F}} W< \alpha.$	
\end{lemma}
\begin{proof}
	There exists a linear transformation $\varphi \in \text{End}_{\mathbb{F}}(V)$ such that $\varphi(W)=(0)$ and $\varphi(v)=v$ for an arbitrary element $v \in z(W)$.
	
	By the assumption on $z$, the image of the linear transformation $z\varphi - \varphi z$  has  dimension less than  $\alpha$.
	
	For an arbitrary element $w \in W $ we have $z(\varphi(w))=0$, $\varphi(z(w))=z(w)$. Therefore, $z(w)=(\varphi z- z\varphi)(w)$. The subspace $z(W)$ lies in the image of $\varphi z- z\varphi$, hence $\dim_{\mathbb{F}} z(W)< \alpha$. If $\dim_{\mathbb{F}} z(W)< \alpha$ and $\dim_{\mathbb{F}} \ker z< \alpha$ 	then $\dim_{\mathbb{F}} W < \alpha.$
\end{proof}

\begin{proof}[Proof of Proposition $\ref{prop5}$] Let $W$ be a maximal subspace of $V$ such that $W \cap z(W)=(0)$ (it exists by Zorn's Lemma). Then $\dim_{\mathbb{F}} W < \alpha$.
	
Choose an element $v \in V \setminus (W \oplus z(W))$. 	Then by maximality of $W$, we have 	$$(W+\mathbb{F} v) \cap z(W+\mathbb{F} v) \neq (0).$$ 	Choose an element $w \in W$ and a scalar $\xi \in \mathbb{F}$ such that 	$$0 \neq z(w+ \xi v ) \in W+\mathbb{F} v.$$
	
\vspace{5pt}
	
\underline{Case 1.} $\xi=0$. Then $0 \neq z(w) \in W + \mathbb{F} v $, which is impossible since $z(W) \cap W=(0)$ and $v \notin W \oplus z(W)$.

\vspace{5pt} 

\underline{Case 2.} If $\xi \neq 0$ then $z(w)+\xi z(v)=w'+ \eta v$, where $w' \in W$, $\eta \in \mathbb{F}$. So,
$$z(v)=\frac{\eta}{\xi}\, v \mod (W \oplus z(W)).$$
We proved that for an arbitrary element $v \in V \setminus (W \oplus z(W))$ there exists a scalar $t_v \in \mathbb{F}$ such that
$$z(v)=t_vv \mod (W \oplus z(W)).$$
Let $v', v''\in V$ be linearly independent modulo $W \oplus z(W)$. Then
\begin{equation}
\label{r1}
z(v')=t_{v'}v' \mod ( W \oplus z(W)),
\end{equation}
\begin{equation}
\label{r2}
z(v'')=t_{v''}v'' \mod (W \oplus z(W)),
\end{equation}
\begin{equation}
\label{r3}
z(v'+v'')=t_{v'+v''}(v'+v'') \mod (W \oplus z(W)).
\end{equation}	
 Subtracting the equalities \eqref{r1} and \eqref{r2} from the 	equality  \eqref{r3} we get
 $$(t_{v'+v''} -t_{v'})v'+ (t_{v'+v''} -t_{v''})v'' \in  (W \oplus z(W)). $$
 In view of the linear independence of $v'$, $v''$ modulo $W \oplus z(W)$ we get
 $$t_{v'}=t_{v''}=t_{v'+v''}.$$
 Hence there exists a scalar $t \in \mathbb{F}$ such that
 $$z(v)=tv \mod  (W \oplus z(W))$$
 for an arbitrary $v \in V$.
 The image of the linear transformation $z-t \cdot \text{Id}_V$ lies in $W \oplus z(W).$ Hence $$\dim_{\mathbb{F}}(z-t \cdot \text{Id}_V) < \alpha \quad \text{and} \quad z-t \cdot \text{Id}_V \in I_{\alpha}.$$ This completes the proof of the proposition. 	
\end{proof}

 \begin{lemma}\label{lemma6} \
 	\begin{enumerate}
 		\item[$(1)$]\label{part1L6}  Let   $\aleph_0 < \alpha \le \dim_{\mathbb{F}} V $. Then
 	$I_{\alpha}=[I_{\alpha},I_{\alpha}].$
 	\item[$(2)$]\label{part2L6} $[I_{\aleph_0},\mathfrak{gl}(V)]=[I_{\aleph_0}, I_{\aleph_0}]$ has co-dimension $1$ in $I_{\aleph_0}$.
 	\end{enumerate}
 \end{lemma}
\begin{proof} 	Let $\varphi \in I_{\alpha}$, $\dim_{\mathbb{F}} \varphi(V) < \alpha$. If 	$\aleph_0 < \alpha$ then there exists a subspace $V'\subset V$ such that $$\varphi(V) \subseteq V' \quad \text{and} \quad  \aleph_0 \le  \dim_{\mathbb{F}} V'< \alpha  .$$ Choose a subspace $V'' \subset V$ such that $V=V' \oplus V''$ is a direct sum. The linear transformation $\varphi$ can be decomposed as $\varphi= \varphi_1+\varphi_2$, where $$\varphi_1: V' \to V' , \ \varphi_1(V'')=(0), \quad \text{and} \quad \varphi_2: V'' \to V' , \ \varphi_2(V')=(0).$$

Let $p$ be a projection from $V$ to $V'$, i.e. $$p|_{V'}=\text{Id}_{V'}, \quad p(V'')=(0).$$ We notice that the images of $\varphi_1$, $\varphi_2$, $p$ lie in $V'.$ Therefore, $\varphi_1, \varphi_2,p \in I_{\alpha}$.
	
By \cite{14_B}, we have $$\varphi_1|_{V'} \in [\mathfrak{gl}(V'),\mathfrak{gl}(V')].$$
Hence $\varphi_1 \in [I_{\alpha}, I_{\alpha}].$ Furthermore,  $$\varphi_2=[p,\varphi_2] \in [I_{\alpha}, I_{\alpha}].$$ We proved that $\varphi \in [I_{\alpha}, I_{\alpha}]$. This completes  the proof of  part (1) of the lemma.

Now, consider the ideal $I_{\aleph_0}$. A linear transformation $\varphi$ lies in $I_{\aleph_0}$ if and only if the subspace $V'=\varphi(V)$ is finite-dimensional.

There exists a finite-dimensional subspace $V'' \subset V$ such that $V'\subseteq V''$ and $\varphi(V'')=V'.$ Let $\text{tr}(\varphi)$ be the trace of the restriction $$\varphi|_{V''}\in \mathfrak{gl}(V'').$$ It is easy to see that
\begin{enumerate}[(i)]
	\item $\text{tr}(\varphi)$ does not depend on a choice of the subspace $V'',$
	 \item $\text{tr}: I_{\aleph_0} \to \mathbb{F}$ is a linear functional,
	 \item $\text{tr}(\varphi)=0$ if and only if $\varphi\in [I_{\aleph_0},I_{\aleph_0}].$
\end{enumerate}
This implies that $[I_{\aleph_0},I_{\aleph_0}]$ has co-dimension 1 in $I_{\aleph_0}$.

It remains to show that $$[I_{\aleph_0},\mathfrak{gl}(V)]=[I_{\aleph_0},I_{\aleph_0}]$$ or equivalently
  $$\text{tr}([I_{\aleph_0},\mathfrak{gl}(V)])=(0).$$
Choose $\varphi \in I_{\aleph_0}$  and $\psi \in \mathfrak{gl}(V).$ Since $\dim_{\mathbb{F}}\varphi( V)<\infty$ the subspace $\ker \varphi$ has a finite co-dimension in $V.$ Hence, there exists a finite-dimensional subspace $V_1\subset V$ such that $\varphi(V)\subset V_1,$ $\varphi(V_1)=\varphi(V)$ and $V=V_1 + \ker \varphi.$ Let $\dim_{\mathbb{F}} V_1 =n.$ Choose a subspace $V_2 \subset \ker \varphi$ such that $$ \ker \varphi = (V_1 \cap \ker \varphi) \oplus V_2$$ is a direct sum. Then $V= V_1 \oplus
  V_2$ is a direct sum. Let $B_1,$ $B_2$ be bases of the subspaces $V_1,$ $V_2$, respectively.  In the basis $B=B_1 \cup B_2$ of the vector space $V$ the linear transformations $\varphi,$ $\psi$ have (infinite) matrices  $$\begin{pmatrix}
  	a & \rvline & \zero \\
  	\hline
    	\zero & \rvline &\zero
  \end{pmatrix}, \quad \begin{pmatrix}
  b_{11} & \rvline & b_{12} \\
  \hline
  b_{21} & \rvline & b_{22}
  \end{pmatrix}$$ respectively, where the blocks $a,$ $b_{11}$ are $n\times n$ matrices. We have $$\left[ \begin{pmatrix}
a & \rvline & \zero \\
\hline
\zero & \rvline &
\zero
\end{pmatrix},
\begin{pmatrix}
b_{11}
& \rvline & b_{12} \\
\hline
b_{21} & \rvline &
b_{22}
\end{pmatrix} \right]= \begin{pmatrix}
[a,b_{11}]
& \rvline & ab_{12} \\
\hline
-b_{21} a & \rvline &
0
\end{pmatrix} .$$

The trace of this infinite matrix is equal to the trace of the $n \times n$ matrix $[a,b_{11}]$, i.e. is equal to 0. This completes the proof of the lemma. 	
\end{proof}

\begin{proof}[Proof of Theorem $\ref{Thm_ideals}$] Let $U$ be a proper  ideal of the Lie algebra $\mathfrak{gl}(V).$ Let $A=\text{End}_{\mathbb{F}}(V).$ Consider the ideal $\text{id}_{A}([U,U]).$ By Jacobson's Theorem, $\text{id}_{A}([U,U])$ equals  $(0),$  $I_{\alpha}$ for some $\alpha$, $\aleph_0 \le \alpha \le \dim_{\mathbb{F}}V,$ or $A.$ We will consider each of these cases separately.
	
\vspace{5pt}
	
\underline{Case 1.} Suppose that $\text{id}_A([U,U])=A$. Then by Lemma \ref{lemma 2}, we have 	$$[\text{id}_A([U,U]),A]=[A,A] \subseteq U.$$ 	By \cite{14_B}, we have $U=\mathfrak{gl}(V),$ which contradicts the assumption that $U$ is proper.

\vspace{5pt}
	
\underline{Case 2.} Suppose that $\text{id}_A([U,U])=(0)$. Then $[U,U]=(0)$.  It is easy to see that the algebra $A$ is prime. Indeed, let $\varphi,$ $\psi\in A$ be nonzero linear transformations $\varphi(v)\neq 0,$ $\psi(w)\neq 0,$ where $v,$ $w\in V.$ There exists a linear transformation $\chi : V \to V$ such that $\chi (\varphi(v))=w.$ Then $$\psi\, \chi\, \varphi(v)=\psi(w)\neq 0,$$ hence $\psi A \varphi\neq (0).$ By Lemma \ref{lemma3}, $U$ lies in the center of the algebra $A$. Hence $U=\mathbb{F} \cdot \text{Id}_V$.
	
\vspace{5pt}
	
\underline{Case 3.} Now, let $\text{id}_A([U,U])=I_{\alpha}$,  	$\aleph_0 \le \alpha \le \dim_{\mathbb{F}}V$. The  ideal $(U+I_{\alpha})/I_{\alpha}$ of the Lie algebra $(A/I_{\alpha})^{(-)}$  is abelian, i.e. 	$$[(U+I_{\alpha})/I_{\alpha}, (U+I_{\alpha})/I_{\alpha}]=(0).$$ 	Lemma \ref{lemma4} implies that the algebra $A/I_{\alpha}$ is prime. By Lemma \ref{lemma3}, the  ideal $(U+I_{\alpha})/I_{\alpha}$ lies in the center $C$ of the algebra $A/I_{\alpha}$. By Proposition \ref{prop5}, the center $C$ is $(\mathbb{F} \cdot \text{Id}_V+I_{\alpha})/I_{\alpha}$. Hence $$U \subseteq \mathbb{F} \cdot \text{Id}_V+I_{\alpha}.$$ 	On the other hand, by Lemma \ref{lemma 2}, $[I_{\alpha}, A] \subseteq U$. We proved that $$[I_{\alpha},A]\subseteq U \subseteq \mathbb{F} \cdot \text{Id}_V +I_{\alpha}.$$

If 	$\aleph_0 < \alpha$ then, by Lemma \ref{lemma6} (1), 	$$I_{\alpha} \subseteq U \subseteq \mathbb{F} \cdot \text{Id}_V +I_{\alpha},$$ 	which implies $U=I_{\alpha}$ or $U= \mathbb{F} \cdot \text{Id}_V +I_{\alpha}.$ 	

Let now $\alpha=\aleph_0$. Then, by Lemma \ref{lemma6} (2), the co-dimension of $[I_{\aleph_0}, A]=[I_{\aleph_0},I_{\aleph_0}]$ in $\mathbb{F}\cdot \text{Id}_V +I_{\aleph_0}$ is equal to $2$. From proved above follows that for an arbitrary subspace $U$,  	$$[I_{\aleph_0}, A] \subseteq U \subseteq  \mathbb{F} \cdot \text{Id}_V +I_{\aleph_0} ,$$ 	  we have $$[U,A] \subseteq [ \mathbb{F} \cdot \text{Id}_V +I_{\aleph_0}, A]=[I_{\aleph_0}, A] \subseteq U.$$ Hence, $U$ is an ideal of the Lie algebra $\mathfrak{gl}(V).$ \end{proof}

Notice that in the case of a countable-dimensional vector space $V$  our description of ideals coincides with the description of ideals in \cite{Hol-inf-Lie}.

\end{document}